\documentclass[11pt]{article}

\usepackage{amsfonts,amsmath,latexsym,color,epsfig}
\newtheorem{theorem}{Theorem}[section]
\newtheorem{lemma}{Lemma}[section]

\newtheorem{conjecture}{Conjecture}
\newtheorem{corollary}{Corollary}[section]

\newenvironment{proof}
      {\medskip\noindent{\bf Proof:}\hspace{1mm}}
      {\hfill$\Box$\medskip}

\input{epsf}

\makeatletter
\def\Ddots{\mathinner{\mkern1mu\raise\p@
\vbox{\kern7\p@\hbox{.}}\mkern2mu
\raise4\p@\hbox{.}\mkern2mu\raise7\p@\hbox{.}\mkern1mu}}
\makeatother

\title{\vspace{-0.7cm}An approximate version of Sidorenko's conjecture}
\author{David Conlon\thanks{St John's College, Cambridge CB2 1TP, United
Kingdom.
E-mail: {\tt
D.Conlon@dpmms.cam.ac.uk}. Research supported by a Junior Research Fellowship
at St John's College.} \and Jacob Fox\thanks{Department of
Mathematics, Princeton University, Princeton, NJ 08544. Email: {\tt
jacobfox@math.princeton.edu}. Research supported by an NSF Graduate
Research Fellowship and a Princeton Centennial Fellowship.} \and
Benny Sudakov\thanks{Department of Mathematics,
UCLA,  Los Angeles, CA 90095. Email: {\tt bsudakov@math.ucla.edu}. Research
supported in part by NSF CAREER award DMS-0812005 and by a
USA-Israeli BSF grant.}}
\date{}
\begin{document}
\maketitle

\begin{abstract}
A beautiful conjecture of Erd\H{o}s-Simonovits and Sidorenko states that if $H$
is a bipartite graph, then the random graph with edge
density $p$ has in expectation asymptotically the minimum number of copies of
$H$ over all graphs of the same order and edge density. This conjecture also
has an equivalent analytic form and has connections to a broad range of topics,
such
as matrix theory, Markov chains, graph limits, and quasirandomness.
Here we prove the conjecture if $H$ has a vertex complete to the other part,
and deduce an approximate version of the conjecture
for all $H$. Furthermore, for a large class of bipartite graphs, we prove a
stronger stability result which answers a question of Chung, Graham, and Wilson
on quasirandomness for these graphs.
\end{abstract}

\section{Introduction}

A fundamental problem in extremal graph theory (see \cite{Bo78} and its
references) is to determine or estimate
the minimum number of copies of a graph $H$ which must be contained in another
graph $G$ of a certain order and size. The special case where one wishes to
determine the minimum number of edges in a graph on $N$ vertices which
guarantee a single copy of $H$ has received particular attention. The case
where $H$
is a triangle was solved by Mantel more than a century ago. This was
generalized to cliques by Tur\'an and the Erd\H{o}s-Stone-Simonovits
theorem determines the answer asymptotically if $H$ is not bipartite. For
bipartite
graphs $H$, a classical result of K\H{o}v\'ari, S\'os, and Tur\'an implies that
$O(N^{2-\epsilon_H})$ edges are sufficient for some
$\epsilon_H>0$, but, despite much effort by researchers, the
asymptotics, and even good estimates for the largest possible $\epsilon_H$, are
understood for relatively few bipartite graphs.

The general problem can be naturally stated in terms of subgraph densities. The
{\it edge density} of a graph $G$ with $N$ vertices and $M$ edges is $M/{N
\choose 2}$. More generally, the {\it $H$-density} of a graph $G$ is the
fraction of all one-to-one mappings from the vertices of $H$ to the vertices of
$G$ which map edges of $H$ to edges of $G$. The general extremal problem asks
for the minimum possible $H$-density over all graphs on $N$ vertices with edge
density $p$. For fixed $H$, the asymptotic answer as $N \to \infty$ is a
function of $p$. Determining this function is a classical problem and
notoriously difficult even in the case where $H$ is the complete graph of order
$r$.
Early results in this case were obtained by Erd\H{o}s, Goodman, Lov\'asz,
Simonovits, Bollob\'as, and Fisher. Recently, Razborov \cite{Ra} using flag
algebras and Nikiforov \cite{Ni} using a combination of combinatorial and
analytic arguments gave an asymptotic answer in the cases $r=3$ and $r=4$,
respectively.

There is a simple upper bound on the minimum $H$-density in terms of the edge
density. Suppose that $H$ has $m$ edges. By taking $G$ to be a random graph
with edge density $p$, it is easy to see that the minimum possible $H$-density
is at most $p^m$. The beautiful conjectures of Erd\H{o}s and Simonovits
\cite{Sim} and Sidorenko \cite{Si3} suggest that this bound is sharp for
bipartite graphs. That is, for any bipartite $H$ there is a $\gamma(H)>0$ such
that the number of copies of $H$ in any graph $G$ on $N$ vertices with edge
density $p >N^{-\gamma(H)}$ is asymptotically at least the same as in the
$N$-vertex random graph with edge density $p$. This is known to be true in a
few very
special cases, e.g., for complete bipartite graphs, trees, even cycles
(see \cite{Si3}) and, recently, for cubes \cite{Ha}.

The original formulation of the conjecture by Sidorenko is in terms of graph
homomorphisms. A {\it homomorphism} from a graph $H$ to a graph $G$ is a
mapping
$f:V(H) \rightarrow V(G)$ such that, for each edge $(u,v)$ of $H$,
$(f(u),f(v))$
is an edge of $G$. Let $h_H(G)$ denote the number of homomorphisms from $H$ to
$G$. We also consider the normalized function $t_H(G)=h_H(G)/|G|^{|H|}$, which
is the fraction of mappings $f:V(H) \rightarrow V(G)$ which
are homomorphisms.
\begin{conjecture}\label{conj1} (Sidorenko) For every bipartite graph $H$ with
$m$ edges and every graph $G$, $$t_H(G) \geq
t_{K_2}(G)^m.$$
\end{conjecture}

Sidorenko's conjecture also has the following nice analytical form. Let $\mu$
be the Lebesgue measure
on $[0,1]$ and let $h(x,y)$ be a bounded, non-negative, symmetric and
measurable function on $[0,1]^2$. Let $H$ be a
bipartite graph with vertices $u_1,\ldots,u_t$ in the first part and vertices
$v_1,\ldots,v_s$ in the second part.
Denote by $E$ the set of edges of $H$, i.e., all the pairs $(i,j)$ such that
$u_i$ and $v_j$ are adjacent, and let
$m = |E|$. The analytic formulation of Sidorenko's conjecture states that
\begin{equation}\label{integralineq}
\int \prod_{(i,j) \in E} h(x_i,y_j) d\mu^{s+t} \geq \left(\int h d\mu^2
\right)^m.\end{equation}
The expression on the left hand side of this inequality is quite common. For
example,
Feynman integrals in quantum field theory, Mayer integrals in statistical
mechanics, and
multicenter integrals in quantum chemistry are of this form (see Section 6 of
\cite{Si91} and its references). Unsurprisingly then, Sidorenko's conjecture
has connections
to a broad range of topics, such as matrix theory \cite{AWM, BR}, Markov
chains \cite{BP, PP}, graph limits \cite{Lo}, and quasirandomness.

The study of quasirandom graphs was introduced by Thomason \cite{Th2} and
Chung, Graham, and Wilson \cite{CGW}. They showed that a large number of
interesting graph properties satisfied by random graphs are all equivalent.
This idea has been quite influential, leading to the study of quasirandomness
in other structures such as hypergraphs \cite{ChGr,Go3}, groups \cite{Go2},
tournaments, permutations, sequences and sparse graphs (see \cite{ChGr3} and it
references), and progress on problems in different areas (see,
e.g., \cite{Con,Go3,Go2}). It is closely related to Szemer\'edi's regularity
lemma and its recent hypergraph generalization and all proofs of Szemer\'edi's
theorem on long arithmetic progressions in dense subsets of the integers use
some notion of quasirandomness. Finally, there is also the fast-growing study
of properties of quasirandom graphs, which has recently
attracted lots of attention both in combinatorics and theoretical computer
science (see, e.g., \cite{KrSu}).

A sequence $(G_n:n=1,2,\ldots)$ of graphs is called {\it quasirandom with
density $p$} (where $0<p<1$) if, for every graph $H$,
\begin{equation}\label{eqquasi}
t_H(G_n)=(1+o(1))p^{|E(H)|}.\end{equation} Note that (\ref{eqquasi}) is
equivalent to saying that the $H$-density of $G_n$ is $(1+o(1))p^{|E(H)|}$,
since the
proportion of mappings from $V(H)$ to $V(G_n)$ which are not one-to-one tends
to $0$ as $|V(G_n)| \to \infty$. This property is equivalent to many other
properties shared by random graphs. One such property is that the edge density
between any two vertex subsets of $G_n$ of linear cardinality is $(1+o(1))p$. A
surprising fact, proved in \cite{CGW}, is that it is enough that
(\ref{eqquasi})
holds for $H=K_2$ and $H=C_4$ for a graph to be quasirandom. That is, a graph
with edge density $p$ is quasirandom with density $p$ if the $C_4$-density is
approximately $p^4$. A question of Chung, Graham, and Wilson \cite{CGW} which
has received considerable attention (see, e.g., \cite{Bo}) asks for which
graphs $H$ is it true that if (\ref{eqquasi}) holds for $K_2$ and $H$, then the
sequence is quasi-random with density $p$. Such a graph $H$ is called {\it
$p$-forcing}. We call $H$ {\it forcing} if it is $p$-forcing for all $p$.
Chung, Graham, and Wilson prove that even cycles $C_{2t}$ and complete
bipartite graphs $K_{2,t}$ with $t \geq 2$ are forcing. Skokan and Thoma
\cite{ST} generalize this result to all complete bipartite graphs $K_{a,b}$
with $a,b \geq 2$.

There are two simple obstacles to a graph being forcing. It is easy to show
that a forcing graph must be bipartite. Further, for any forest $H$,
(\ref{eqquasi}) is satisfied for any sequence of nearly regular graphs of edge
density tending to $p$. The property of being nearly regular is not as strong
as being quasirandom. Hence, a forcing graph must be bipartite and have at
least one cycle. Skokan and Thoma \cite{ST} ask whether these properties
characterize the forcing graphs.
We conjecture the answer is yes and refer to it as the {\it forcing
conjecture}.

\begin{conjecture}\label{forcingconj}
A graph $H$ is forcing if and only if it is bipartite and contains a cycle.
\end{conjecture}

It is not hard to see that the forcing conjecture is stronger than Sidorenko's
conjecture, and it further gives a stability result for Sidorenko's conjecture.
A {\it stability} result not only characterizes the extremal graphs for an
extremal problem, but also shows that if a graph is close to being optimal for
the extremal problem, then it is close in a certain appropriate metric to an
extremal graph. In recent years, there has been a great amount of research done
toward proving stability results in extremal combinatorics. The forcing
conjecture implies that if $H$ is bipartite with $m$ edges and contains a
cycle, then $G$ satisfies $t_H(G)$ is close to $t_{K_2}(G)^m$ if and only if it
is quasirandom with density $t_{K_2}(G)$.

As consequences of the following theorem, we prove Sidorenko's conjecture and
the forcing conjecture for a large class of bipartite graphs.

\begin{theorem}\label{main1}
Let $H$ be a bipartite graph with $m$ edges which has $r \geq 1$ vertices in
the first part complete to the second part, and the minimum degree
in the first part is at least $d$. Then $$t_{H}(G) \geq
t_{K_{r,d}}(G)^{\frac{m}{rd}}.$$
\end{theorem}

From Theorem \ref{main1}, by taking $r=1$ and $d=1$, we have the following
corollary.

\begin{theorem}\label{main}
Sidorenko's conjecture holds for every bipartite graph $H$ which has a vertex
complete to the other part.
\end{theorem}

From Theorem \ref{main}, we may easily deduce an approximate version of
Sidorenko's conjecture for all graphs. For a connected bipartite graph $H$
with parts $V_1,V_2$, define the bipartite graph $\bar{H}$ with parts
$V_1,V_2$ such that $(v_1,v_2) \in V_1 \times V_2$ is an edge of $\bar{H}$ if
and only if it is not an edge of $H$. Define the {\it width} of $H$ to be the
minimum degree  of $\bar{H}$. If $H$ is not connected, the width of $H$ is
the sum of the widths of the connected components of $H$. Note that the width
of a connected bipartite graph is $0$ if and only if it has a vertex that is
complete to the other part. Also, the width of a bipartite graph with $n$
vertices is at most $n/2$.

\begin{corollary}\label{cor}
If $H$ is a bipartite graph with $m$ edges and width $w$, then $t_H(G) \geq
t_{K_2}(G)^{m+w}$ holds for every graph $G$.
\end{corollary}

We also obtain the following result from Theorem \ref{main1}.

\begin{theorem}\label{main2}
The forcing conjecture holds for every bipartite graph $H$ which has two
vertices in one part complete to the other part, which has at least two
vertices.
\end{theorem}

Sidorenko further conjectured that the assumption that $h$ is
symmetric in (\ref{integralineq}) can be dropped. This has the following
equivalent discrete version. For bipartite graphs $H=(V_1,V_2,E)$ and
$G=(U_1,U_2,F)$ where $H$ has $m=|E|$ edges and $G$ has edge density
$p=\frac{|F|}{|U_1||U_2|}$ between its parts, the density of mappings $f:V(H)
\to V(G)$ with $f(V_i) \subset U_i$ for $i=1,2$ that are homomorphisms is at
least $p^m$. It is not hard to check
that the proofs of Theorems \ref{main1}, \ref{main}, and \ref{main2} and
Corollary \ref{cor} can be extended to prove stronger asymmetric versions. We
leave the details to the interested reader.

Although several authors (e.g., Sidorenko \cite{Si91} and Lov\'asz \cite{Lo})
suggested that one might need an analytic approach to attack Conjecture
\ref{conj1}, our proof of Theorem \ref{main1} given in the next section uses
simple combinatorial tools. We conclude with a discussion of some related
problems.

\section{Proof of Theorem \ref{main1}}

We begin with some simple observations. For a vertex $v$ in a graph $G$, the
{\it neighborhood} $N(v)$ is the set of vertices adjacent
to $v$. For a sequence $S$ of vertices of a graph $G$, the {\it common
neighborhood} $N(S)$
is the set of vertices adjacent to every vertex in $S$. The identity
\begin{eqnarray}\label{simpleob}
h_{K_{a,b}}(G)=\sum_{T} |N(T)|^a,
\end{eqnarray} where the sum is over all sequences $T$ of $b$ vertices of $G$,
follows by counting homomorphisms of $K_{a,b}$ which fix the second part. Here
we allow sequences of vertices to include repeated vertices.

The previous approaches toward proving Sidorenko's conjecture mainly used
clever applications of H\"{o}lder's inequality. We propose a new approach using
a
probabilistic technique known as dependent random choice (for more details,
see, e.g., the survey \cite{FoSu}). The first attempt
to use this technique to estimate subgraph densities was made in \cite{FoSu2}.
Roughly, the idea is that most small subsets of the neighborhood of a random
subset of vertices have large common neighborhood. Our proof uses an equivalent
counting version.

Before going into the details of the proof of Theorem \ref{main1}, we first
give a brief outline of the proof idea in the case $r=1$ and $d=1$. Suppose $u$
is a vertex in the bipartite graph $H=(V_1,V_2,E)$ on $n$ vertices which is
complete to $V_2$, and $G$ is a graph on $N$ vertices with edge density $p$.
The bulk of the proof is geared toward proving a seemingly weaker result, Lemma
\ref{importantstep}, which shows that the bound in Theorem \ref{main1} is tight
apart from a positive constant factor which only depends on $n=|H|$. To obtain
this result, we use dependent random choice to show that (see Lemma \ref{DRC})
an average vertex $v$ of $G$ (weighted by its degree) has the property that
almost all small subsets $S$ of $N(v)$ satisfy $|N(S)| \geq c_np^{|S|}N$,
which, apart from the factor $c_n$, is the expected size of the common
neighborhood of a subset of vertices of size $|S|$ in the random graph
$G(N,p)$. We will give a lower bound on the number of homomorphisms $f:V(H)
\rightarrow V(G)$ from $H$ to $G$ as follows. We first pick $f(u)=v$ so that
almost all small subsets of vertices in $N(v)$ have large common neighborhood.
Having picked $f(u)=v$, we then randomly pick a sequence of $|V_2|$ vertices
from $N(v)$ to be $f(V_2)$. With large probability, for all subsets $S \subset
f(V_2)$, we have $|N(S)| \geq c_np^{|S|}N$. For any vertex $u' \in V_1
\setminus \{u\}$, we can pick $f(u')$ to be any vertex in the common
neighborhood of $S=f(N(u')) \subset f(V_2)$. To summarize, we get a lower bound
on the number
of homomorphisms from $H$ to $G$ by first choosing the image of $u$, then the
image of $V_2$, and finally the image of the remaining vertices in $V_1$. The
homomorphism count is within a positive constant factor, depending only on $n$,
of $p^mN^n$, which is asymptotically the expected homomorphism count in $G(N,p)$. We complete
the proof of Theorem \ref{main1} by using a tensor power trick to
get rid of the factor depending on $n$.

\begin{lemma} \label{DRC}
Let $G$ be a graph with $N$ vertices and $d$, $n$, and $r$ be positive
integers with $d \leq n$. Call a sequence $S$ of $k$ vertices of $G$ rare if
$|N(S)| \leq (2n)^{-2n}t_{K_{r,d}}(G)^{\frac{k}{rd}}N$. Call a sequence
$T=(v_1,\ldots,v_r)$ of $r$ vertices
bad with respect to $k$ if the number of rare sequences of $k$ vertices in
$N(T)$ is at least
$\frac{1}{2n}|N(T)|^k$. Call $T$ good if, for all $d \leq k \leq n$, it is not
bad with respect to $k$. Then the sum of $|N(T)|^d$ over all good sequences $T$
is at least
$h_{K_{r,d}}(G)/2$.
\end{lemma}
\begin{proof}
We write $T \sim k$ to denote that $T$ is bad with respect to $k$. Let $X_k$
denote the number of pairs $(T,S)$ with $S$ a rare sequence of $k$ vertices and
$T$ a
sequence of $r$ vertices which are adjacent to every vertex in $S$. For a given
sequence $S$, the number of pairs $(T,S)$ with $T$ a sequence of $r$ vertices
adjacent to every vertex in $S$ is $|N(S)|^r$. As the number of sequences $S$
of $k$ vertices is $N^k$, we have, by summing over rare $S$,
\begin{equation}\label{eqfirst}X_k=\sum_{S~\textrm{rare}}|N(S)|^r \leq N^k
\cdot \left((2n)^{-2n}t_{K_{r,d}}(G)^{\frac{k}{rd}}N\right)^r
=(2n)^{-2nr}t_{K_{r,d}}(G)^{\frac{k}{d}}N^{k+r}.\end{equation}
Of course, $X_k$ is at least the number of such pairs $(T,S)$ with $T$ having
the additional property that it is bad with respect to $k$. Hence,
by summing over such $T$, we have

\begin{equation}\label{eqsecond}X_k \geq \sum_{T, T \sim k}\frac{1}{2n}
| N(T)|^k \geq \frac{1}{2n} N^r\left(\sum_{T, T \sim k}
| N(T)|^d/N^r\right)^{k/d} = \frac{1}{2n} N^{r-rk/d}\left(\sum_{T, T \sim k}
| N(T)|^d\right)^{k/d}.\end{equation}
The second inequality follows from convexity of the function $f(x)=x^{k/d}$
together with the fact that
there are at most $N^r$ sequences $T$. From
(\ref{eqfirst}) and (\ref{eqsecond}) and simplifying, we get

\begin{eqnarray*}\sum_{T,T\sim k} |N(T)|^d & \leq &
\left(2nN^{\frac{rk}{d}-r}X_k\right)^{d/k} \leq
\left(2nN^{\frac{rk}{d}-r}(2n)^{-2nr}t_{K_{r,d}}(G)^{\frac{k}{d}}N^{k+r}\right)^{d/k}
\\ & = &(2n)^{(1-2nr)d/k}t_{K_{r,d}}(G)N^{d+r} \leq
\frac{1}{2n}t_{K_{r,d}}(G)N^{d+r}=\frac{1}{2n}h_{K_{r,d}}(G).\end{eqnarray*}
Hence, using (\ref{simpleob}), we have $$\sum_{T~\textrm{good}} |N(T)|^d \geq
\sum_{T} |N(T)|^d - \sum_{k=1}^n
\sum_{T,T\sim k} |N(T)|^d \geq h_{K_{r,d}}(G)-n \cdot
\frac{1}{2n}h_{K_{r,d}}(G) = h_{K_{r,d}}(G)/2.$$
\end{proof}

The bound on $|N(S)|$ in the definition of a rare sequence $S$ in the above
lemma is quite natural. Indeed, in the case $G=G(N,p)$, $t_{K_{r,d}}(G)\approx
p^{rd}$ and a sequence $S$ of order $k$ is rare if $|N(S)| \leq
c_nt_{K_{r,d}}(G)^{\frac{k}{rd}}N \approx c_np^kN$ with $c_n=(2n)^{-2n}$,
which, apart from the constant factor $c_n$, is roughly the size of the common
neighborhood of every subset of order $k$.

A {\it hypergraph} $\mathcal{H}=(V,E)$ consists of a set $V$ of vertices and a
set $E$ of edges, which are subsets of $V$.
A {\it strongly directed hypergraph} $\mathcal{H}=(V,E)$ consists of a set $V$
of vertices and a set $E$ of edges, which are sequences of vertices in $V$.

\begin{lemma} \label{Embedding}
Suppose $\mathcal{H}$ is a hypergraph with $h$ vertices and at most $e$ edges
such that each edge has size at least $d$,
and $\mathcal{G}$ is a strongly directed hypergraph on $N$ vertices with the
property that for
each $k$, $d \leq k \leq h$, the number of sequences of $k$ vertices of
$\mathcal{G}$ that do not form an edge of $\mathcal{G}$ is at most
$\frac{1}{2e}N^k$. Then the number of homomorphisms from $\mathcal{H}$ to
$\mathcal{G}$ is at least $\frac{1}{2}N^h$.
\end{lemma}
\begin{proof}
Consider a random mapping from the vertices of $\mathcal{H}$ to the vertices of
$\mathcal{G}$. The probability that a given edge of $\mathcal{H}$ does not map
to an edge of $\mathcal{G}$ is at most $\frac{1}{2e}$. By the union bound, the
probability that there is an edge of $\mathcal{H}$ that does not map to an edge
of $\mathcal{G}$ is at most $e \cdot \frac{1}{2e}=1/2$. Hence, with probability
at least $1/2$, a random mapping gives a homomorphism, so there are at least
$\frac{1}{2}N^h$ homomorphisms from $\mathcal{H}$ to $\mathcal{G}$.
\end{proof}

\begin{lemma}\label{importantstep}
Let $H=(V_1,V_2,E)$ be a bipartite graph with $n$ vertices and $m$ edges such
that there are $r$ vertices $u_1,\ldots,u_r \in V_1$ which are adjacent to all
vertices in $V_2$, and the minimum degree in $V_1$ is at least $d$. Then, for
every graph $G$, $$t_{H}(G) \geq (2n)^{-2n^2}t_{K_{r,d}}(G)^{\frac{m}{rd}}.$$
\end{lemma}
\begin{proof}
Let $N$ denote the number of vertices of $G$. Let $n_i=|V_i|$ for $i \in
\{1,2\}$. We will give a lower bound on the number
of homomorphisms $f:V(H) \rightarrow V(G)$ that map $u_1,\ldots,u_r$ to a good
sequence $T=(v_1,\ldots,v_r)$ of $r$ vertices of
$G$. Suppose we have already picked $f(u_i)=v_i$ for $1 \leq  i \leq r$. Let
$\mathcal{H}$ be the
hypergraph with vertex set $V_2$, where $S \subset V_2$ is an edge of
$\mathcal{H}$ if there is a vertex $w \in V_1 \setminus \{u_1,\ldots,u_r\}$
such that
$N(w)=S$. The number of vertices of $\mathcal{H}$ is $n_2$, which is at most
$n$, and the number of edges of $\mathcal{H}$ is at most $n_1-r$, which is at
most $n$.
Let $\mathcal{G}$ be the strongly directed hypergraph on $N(T)$, where a
sequence $R$ of $k$
vertices in $N(T)$ is an edge of $\mathcal{G}$ if $|N(R)| \geq
(2n)^{-2n}t_{K_{r,d}}(G)^{\frac{k}{rd}}N$. Since $T$ is good, for each $k$, $d
\leq k \leq n$, the
number of sequences of $k$ vertices of $\mathcal{G}$ that are not edges of
$\mathcal{G}$ is at most $\frac{1}{2n}|N(T)|^k$.
Hence, by Lemma \ref{Embedding}, there are at least $\frac{1}{2}|N(T)|^{n_2}$
homomorphisms $g$ from $\mathcal{H}$ to
$\mathcal{G}$. Pick one such homomorphism $g$, and let $f(x)=g(x)$ for $x \in
V_2$. By construction, once we have picked $T$ and $f(V_2)$, there are at
least $(2n)^{-2n}t_{K_{r,d}}(G)^{\frac{|N(w)|}{rd}}N$ possible choices for
$f(w)$ for each vertex $w
\in V_1 \setminus \{u_1,\ldots,u_r\}$. Hence, the number of homomorphisms from
$H$ to $G$
satisfies
\begin{eqnarray*}
t_H(G)N^n & \geq & \sum_{T~\textrm{good}}\frac{1}{2}|N(T)|^{n_2}\prod_{w \in
V_1
\setminus \{u_1,\ldots,u_r\}}(2n)^{-2n}t_{K_{r,d}}(G)^{\frac{|N(w)|}{rd}}N \\ &
= &
\frac{1}{2}(2n)^{-2n(n_1-r)}t_{K_{r,d}}(G)^{\frac{m-rn_2}{rd}}N^{n_1-r}\sum_{T~\textrm{good}}|N(T)|^{n_2}
\\ & \geq &
\frac{1}{2}(2n)^{-2n(n_1-r)}t_{K_{r,d}}(G)^{\frac{m-rn_2}{rd}}N^{n_1-r}N^r\left(\sum_{T~\textrm{good}}|N(T)|^d/N^r
\right)^{n_2/d}
 \\ & = &
\frac{1}{2}(2n)^{-2n(n_1-r)}t_{K_{r,d}}(G)^{\frac{m-rn_2}{rd}}N^{n_1-rn_2/d}\left(\sum_{T~\textrm{good}}|N(T)|^d
\right)^{n_2/d}
 \\ & \geq &
\frac{1}{2}(2n)^{-2n(n_1-r)}t_{K_{r,d}}(G)^{\frac{m-rn_2}{rd}}N^{n_1-rn_2/d}\left(h_{K_{r,d}}(G)/2\right)^{n_2/d}
 \\ & =  &
2^{-1-n_2/d}(2n)^{-2n(n_1-r)}t_{K_{r,d}}(G)^{\frac{m}{rd}}N^{n_1+n_2}
 \\ & \geq & (2n)^{-2n^2}t_{K_{r,d}}(G)^{\frac{m}{rd}}N^{n}.
\end{eqnarray*}
In the first equality, we use $\sum_{w \in
V_1
\setminus \{u_1,\ldots,u_r\}}|N(w)|=m-rn_2$, which follows from the fact that
the vertices $u_1,\ldots,u_r$ each have degree $n_2$.
The second inequality uses the convexity of the function $q(x)=x^{n_2/d}$
together with the fact that there are at most
$N^r$ sequences $T$. The
third inequality follows by Lemma \ref{DRC}, and the third equality follows
from substituting $h_{K_{r,d}}(G)=t_{K_{r,d}}(G)N^{d+r}$. Dividing by $N^n$, we
get the desired inequality. \end{proof}

We next complete the proof of Theorem \ref{main1} by improving the inequality
in the
previous lemma using a tensor power trick. This technique was used by Alon and
Ruzsa \cite{AR} to give an elementary proof of Sidorenko's conjecture for
trees, which implies the Blakley-Roy matrix inequality \cite{BR}. This
technique has also been used in many other areas, and Tao \cite{Ta} has
collected a number of these applications. The tensor product $F \times G$ of
two graphs $F$ and $G$ has vertex set $V(F) \times V(G)$ and any two
vertices $(u_1,u_2)$ and $(v_1,v_2)$ are adjacent in $F \times G$ if and only
if $u_i$ is adjacent with $v_i$ for $i \in \{1,2\}$.
Let $G^1=G$ and $G^s=G^{s-1} \times G$. Note that $t_H(F \times G)=t_H(F)
\times t_H(G)$ for all $F, G, H$.

{\bf Proof of Theorem \ref{main1}:} Suppose, for contradiction, that there is a
graph $G$ such that $t_{H}(G)<t_{K_{r,d}}(G)^{\frac{m}{rd}}$. Let
$c=t_H(G)/t_{K_{r,d}}(G)^{\frac{m}{rd}}<1$. Let $s$
be such that $c^s<(2n)^{-2 n^2}$. Then
$$t_{H}(G^s)=t_H(G)^s=c^st_{K_{r,d}}(G)^{\frac{ms}{rd}}=c^st_{K_{r,d}}(G^s)^{\frac{m}{rd}}<(2n)^{-2n^2}t_{K_{r,d}}(G^s)^{\frac{m}{rd}}.$$
However, this contradicts Lemma \ref{importantstep} applied to $H$ and $G^s$.
This completes the proof.
{\hfill$\Box$\medskip}

{\bf Proof of Corollary \ref{cor}:}
First assume $H$ is connected. Let $H'$ be obtained from $H$ by making
the vertex $u$ of minimum degree in $\bar{H}$ complete to the other part. Note
that $H$ is a subgraph of $H'$ and $H'$ has exactly $w$ more edges than $H$.
Hence, by Theorem \ref{main}, $t_H(G) \geq t_{H'}(G) \geq t_{K_2}(G)^{m+w}$. If
$H$ is not connected, letting $H_1,\ldots,H_r$ denote the connected components
of $H$, we have
$t_H(G)=\prod_{i=1}^r t_{H_i}(G)$, and the corollary easily follows from the
case where $H$ is connected.
{\hfill$\Box$\medskip}

{\bf Proof of Theorem \ref{main2}:}
Let $H=(V_1,V_2,E)$ be a bipartite graph with $n$ vertices, $m$ edges such
that there are $2$ vertices in $V_1$ which are adjacent to all
vertices in $V_2$ and $|V_2| \geq 2$. We may suppose $H$ has no isolated
vertices as they do not affect
homomorphism density counts. Fix $0<p<1$. Suppose $G$ has $N$ vertices and
$(1+o(1))pN^2/2$ edges so that $t_{K_2}(G)=(1+o(1))p$, and $t_H(G) =
(1+o(1))p^m$.

We first prove the case of Theorem \ref{main2} where the minimum degree of
$H$ is at least $2$. By Theorem \ref{main1} with $r=d=2$, we have $t_H(G)
\geq t_{K_{2,2}}(G)^{m/4}$. From the above bounds on $t_{H}(G)$, we get
$t_{K_{2,2}}(G) \leq (1+o(1))p^4$. Also, since Sidorenko's conjecture
holds for $K_{2,2}$, we have $t_{K_{2,2}}(G)\geq (1+o(1))p^4$. Thus,
$t_{K_{2,2}}(G) = (1+o(1))p^4$. Since $K_{2,2}$ is forcing, this implies
$G$ is quasirandom with density $p$, and hence $H$ is forcing.

Suppose now that $H$ has $s \geq 1$ vertices of degree $1$.
Then all of these vertices are in $V_1$. Let $H'$ be the induced subgraph of
$H$ with $n-s$
vertices and $m-s$ edges obtained by deleting all degree $1$ vertices.
Note that $H'$ has minimum degree at least $2$ and also has at least two
vertices in one
part complete to the other part. By the above argument, $H'$ is forcing.
We can apply Theorem \ref{main1} with $r=2$ and $d=1$ to get $t_H(G) \geq
t_{K_{1,2}}(G)^{m/2}$. From the bounds on $t_{H}(G)$, we get
$t_{K_{1,2}}(G) \leq (1+o(1))p^2$. This inequality implies the following
claim:

\noindent {\bf Claim:} All but $o(N)$ vertices of $G$ have degree at least
$(1+o(1))pN$.

\begin{proof}
Suppose, instead, that there are $\epsilon N$ vertices with degree less than
$(1 - \epsilon)pN$. For each vertex $v$, let $\delta(v) = |N(v)| - pN$. Since
$t_{K_2}(G)=(1+o(1))p$, we have $\sum_v \delta(v) = o(p N^2)$. Moreover, since
$\delta(v) \leq - \epsilon p N$ for $\epsilon N$ vertices, $\sum_v \delta^2(v)
\geq \epsilon^3 p^2 N^3$. Therefore,
\begin{eqnarray*}
h_{K_{1,2}}(G) & = & \sum_{v}|N(v)|^2 = \sum_v (pN + \delta(v))^2\\
& = & p^2 N^3 + 2 p N \sum_v \delta(v) + \sum_v \delta^2(v)\\
& \geq & (1 + o(1) + \epsilon^3) p^2 N^3.
\end{eqnarray*}
Thus $$t_{K_{1,2}}(G) = N^{-3}h_{K_{1,2}}(G) \geq
\left(1+o(1)+\epsilon^3\right)p^2,$$ which contradicts
$t_{K_{1,2}}(G) \leq (1+o(1))p^2$ and verifies the claim.
\end{proof}

Any set $X$ of $o(N)$ vertices of $G$, and in particular the $o(N)$
vertices of least degree, can only change the subgraph density count by
$o(1)$ for any fixed subgraph. Indeed, suppose $F$ has $k$ vertices, then
the number of mappings from $V(F)$ to $V(G)$ whose image $f(F)$ has
nonempty intersection with $X$ is $N^k-(N-|X|)^k=o(N^{k})$, and hence the
fraction of mappings whose image has nonempty intersection with $X$ is
$o(1)$. Every homomorphism from $H'$ to $G$ whose image does not intersect
the $o(N)$ vertices of least degree can be extended to at least
$\left((1+o(1))pN\right)^s$ homomorphisms from $H$ to $G$ by picking the
images of the $s$ vertices of $H$ of degree $1$. By normalizing, we get
$t_{H}(G) \geq (1+o(1))p^st_{H'}(G)-o(1)$. Comparing the bounds on
$t_H(G)$, we get $t_{H'}(G) \leq (1+o(1))p^{m-s}$. Since $H'$ satisfies
Sidorenko's conjecture, $t_{H'}(G) \geq (1+o(1))p^{m-s}$, and hence
$t_{H'}(G) = (1+o(1))p^{m-s}$. As $H'$ is forcing, this implies $G$ is
quasirandom with density $p$, and hence $H$ is also forcing.
{\hfill$\Box$\medskip}

\section{Concluding Remarks}

$\bullet$ The classical theorem of Ramsey states that every $2$-edge-coloring
of a sufficiently large complete graph $K_N$ contains at least one
monochromatic copy of a given graph $H$. Let $c_{H,N}$ denote the fraction of
copies of $H$ in $K_N$ that must be monochromatic in any $2$-edge-coloring. By
an averaging argument, $c_{H,N}$ is a monotone increasing function in $N$, and
therefore has a limit $c_H$ as $N \to \infty$. The constant $c_H$ is known as
the {\it Ramsey multiplicity constant} for the graph $H$. For $H$ with $m$
edges, the uniform random $2$-edge-coloring of $K_N$ shows that $c_H \leq
2^{1-m}$. Erd\H{o}s \cite{Er} and, in a more general form, Burr and Rosta
\cite{BuRo} conjectured that this bound is tight. However, Thomason \cite{Th}
proved that these conjectures are already false for $K_4$. In fact, it is false for almost all
graphs
\cite{JST}, and there are graphs $H$ with $m$ edges for which $c_{H} \leq
m^{-\frac{m}{2}+o(m)}$, showing that the bound can be far from tight \cite{F}.
The situation for bipartite graphs is very different as Sidorenko's conjecture
implies the Ramsey multiplicity conjecture holds for bipartite graphs. Thus,
for a large class of bipartite graphs, this conjecture is implied by Theorem
\ref{main} and further a stability result follows from Theorem \ref{main2}. For
all bipartite graphs, an approximate version of this conjecture follows from
Corollary \ref{cor}.

Despite the fact that the conjectures of Erd\H{o}s and Burr-Rosta are false for $K_4$,
Franek and R\"odl \cite{FrRo} proved that a local version is true in this case. They showed that a small
perturbation of a quasirandom coloring will
not give a counterexample to these conjectures. In terms of graphons, in which
we consider the limit of the graphs of one color for a sequence of colorings,
the uniform graphon, which is the limit of a sequence of
quasirandom graphs with density $1/2$, is a local minimum for the density of
monochromatic copies of the graph $K_4$, but not always the global minimum. Very
recently,
Lov\'asz \cite{Lo2} proved the analogous local version of Sidorenko's
conjecture.

$\bullet$ It is tempting to conjecture that the products of
integrals in the left hand side of (\ref{integralineq}) are always
non-negatively correlated.
Equivalently, if $H$ is a bipartite graph and $H_1,\ldots,H_r$ are subgraphs
which edge-partition $H$, then for every graph $G$ we have
\begin{equation}\label{nonnegativecor}
t_{H}(G) \geq \prod_{i=1}^r t_{H_i}(G).
\end{equation} Sidorenko's conjecture is the case $H_i=K_2$ for each $i$, and
the forcing conjecture would follow by taking $H_1$ to be a cycle in $H$ and
all other $H_i=K_2$. Unfortunately, (\ref{nonnegativecor}) does not hold in
general. As noted in \cite{Si91,Si3}, a counterexample to this inequality with
$H=P_3$,
$H_1=P_2$, and $H_2=P_1$, where $P_i$ denotes the path with $i$ edges, was
found in 1966 by London (see \cite{Si91}).

$\bullet$ Let $\|h\|_H$ be the integral on the left hand side of
(\ref{integralineq}) raised to the power $1/|E(H)|$. Lov\'asz asked for which
graphs $H$ is $\|\cdot\|_H$ a norm. When $H$ is an even
cycle, for example, they are the classical Schatten-von Neumann norms. Hatami
\cite{Ha} showed
that if $\|\cdot\|_H$ is a norm then $H$ has the following property. For all
subgraphs $H'$ of $H$ and all $h$,
\begin{equation}\label{partialorder}\|h\|_H \geq \|h\|_{H'}.\end{equation}
Therefore, if $\|\cdot\|_H$ is a norm, then, taking $H'=K_2$, we see that
Sidorenko's conjecture holds for $H$, and if $H$ also has a cycle, then, taking
$H'$ to be that cycle, we see that $H$ is forcing. Property
(\ref{partialorder}) implies that the non-negative correlation inequality
(\ref{nonnegativecor}) holds for $H$. Hence, property (\ref{partialorder}) does
not hold for $H=P_3$. Hatami proved that if $\|\cdot\|_H$ is a norm, then
the degrees of any two vertices in the same part of $H$ are
equal. This shows, in combination with Theorem \ref{main}, that the class of graphs for which $\|\cdot\|_H$
is a norm is considerably smaller than that for which Sidorenko's
conjecture holds. In the positive direction, Hatami proves that if $H$ is a
cube or a complete bipartite graph, then $\|\cdot\|_H$ is a norm.

This discussion leads to the following general question: for which pairs of
graphs $H,H'$ do we have $\|\cdot\|_H \geq \|\cdot\|_{H'}$? Sidorenko's
conjecture is that if $H'=K_2$, then this inequality holds if and only if $H$
is bipartite. Theorem \ref{main1} shows that this inequality holds if
$H'=K_{r,d}$ and $H$ has $r$ vertices in the first part complete to the second
part and the minimum degree in the first part is at least $d$.  The case where
$H$ and $H'$ are trees of the same size is studied in \cite{Si91}.
From the fact the inequality holds if $H$ is a complete bipartite graph and
$H'=K_2$, one may naturally guess that the inequality holds if $H$ is a
complete tripartite graph and $H'=K_3$, but a counterexample is given in
\cite{ShYu}. It may be the case that the general problem is undecidable. There
are other similar questions involving linear inequalities between graph homomorphism densities which are
undecidable by a reduction to Hilbert's 10th Problem (see  \cite{HN}).

$\bullet$ Call a family $\mathcal{F}$ of graphs $p$-forcing if, whenever
$t_F (G_n) = (1 +o(1)) p^{|E(F)|}$ for all $F \in \mathcal{F}$, the sequence
$(G_n: n = 1, 2, \ldots)$ is $p$-quasirandom. If $\mathcal{F}$ is $p$-forcing
for all
$p$, we simply say that the family $\mathcal{F}$ is forcing. Note that the
statement that a graph $H$ is forcing is equivalent to the family
$\mathcal{F}=\{K_2,H\}$ being forcing.
The first examples of forcing families not involving any bipartite graphs were
given in \cite{CHPS}. For example,
let $L_3$ be the line graph of the cube, that is, it has $12$ vertices
corresponding to the
edges of the cube and two vertices are connected if and only if the
corresponding edges meet.
Then the pair consisting of $L_3$ and the triangle $K_3$ is forcing. It would
be interesting
to extend this result and determine what other graphs may be coupled with
the triangle to give a forcing pair.

$\bullet$ We conclude by mentioning that counterexamples are given in
\cite{Si3} to the natural generalization of inequality (\ref{integralineq}) to
functions of more than two variables. That is, the hypergraph analogue of
Sidorenko's conjecture is false.

\end{document}